\documentclass[12pt]{article}

\usepackage[T1]{fontenc}

\usepackage{amsmath}
\usepackage{amsthm}
\usepackage{amssymb}

\usepackage{csquotes}
\MakeOuterQuote{"}

\usepackage{xargs}

\usepackage{enumitem}
\usepackage{calc}
\usepackage{comment}

\usepackage{xcolor}
\usepackage{tikz}

\usepackage[small, pagestyles]{titlesec}
\usepackage[small, it]{caption}

\usepackage[square,comma,numbers,sort&compress]{natbib}

\usepackage[margin=1in,letterpaper,portrait]{geometry}

\theoremstyle{plain}

\newtheorem{theorem}{Theorem}[section]
\newtheorem{proposition}[theorem]{Proposition}

\newtheorem*{proposition*}{Proposition}

\newtheorem{corollary}[theorem]{Corollary}

\newtheorem{observation}[theorem]{Observation}

\newtheorem*{theorem*}{Theorem}

\theoremstyle{definition}

\tikzstyle{every path}=[
    thick
]

\tikzstyle{every node}=[
    minimum width=5pt,
    inner sep=1pt,
]

\tikzstyle{hollow}=[
    shape=circle,
    fill=white,
    node contents={},
    outer sep={0pt},
    draw
]

\tikzstyle{solid}=[
    shape=circle,
    fill=black,
    node contents={},
    outer sep={0pt},
    draw
]

\newcommand{\plotperm}[2][v]{%
  \foreach \j [count=\i] in {#2} {%
    \node[draw] (#1\j) at (\i,\j) [solid] {};
  }%
}

\newcommandx*{\labelnode}[3][1=45]{%
	\pgfgettransformentries
		\myxscale\myxslant\myyslant\myyscale\myxshift\myyshift
	\node at ([shift=({#1}:{0.1/\myxscale})]#2)
		[anchor={180+#1}] {#3};%
}

\usepackage[bookmarks]{hyperref}

\hypersetup{
  colorlinks=true,
  linkcolor=black,
  anchorcolor=black,
  citecolor=black,
  urlcolor=black,
  pdfpagemode=UseThumbs,
  pdftitle={Cyclomatic numbers and permutations},
  pdfsubject={Combinatorics},
  pdfauthor={Bridget Eileen Tenner and Vincent Vatter},
}

\title{Cyclomatic numbers and permutations}

\author{
  Bridget Eileen Tenner%
  \footnote{
    Department of Mathematical Sciences,
    DePaul University,
    Chicago, Illinois.
    \texttt{bridget@math.depaul.edu}\\
    Research partially supported by NSF Grant DMS-2054436 and by a University Research Council Competitive Research Leave from DePaul University.
  }
  \quad
  and
  \quad
  Vincent Vatter%
  \footnote{
    Department of Mathematics,
    University of Florida,
    Gainesville, Florida.
    \texttt{vatter@ufl.edu}
  }
}

\date{}

\begin{document}
\maketitle


\begin{abstract}
We show that several apparently different aspects of a permutation are all tied to a single quantity, the cyclomatic number of its inversion graph. Every reduced word for the permutation orders the edges of the inversion graph one at a time, with the edges from first-occurrence letters forming a spanning forest and the edges from repeated letters accounting for the rest; the number of repeated letters is therefore the cyclomatic number. The excess of permutation cycles over sum components is also at most this quantity, and it follows that the gap between the Coxeter and reflection lengths is at least the cyclomatic number and at most twice it. When the inversion graph is a forest, these results unify classical characterizations of the boolean permutations due to Edelman, to Tenner, and to Petersen and Tenner. We also give a new proof that every connected acyclic inversion graph is a caterpillar.
\end{abstract}

Several apparently disparate results about permutations, spread across the literature on Coxeter groups, reduced words, permutation patterns, and cycle structure, admit a common explanation: they are all governed by the inversion graph. This paper develops that unifying viewpoint.

Our quantitative results are the following. For a permutation $\pi$ with inversion graph~$G_\pi$, the cyclomatic number $\beta(G_\pi)$ (the dimension of the cycle space of~$G_\pi$) equals the number of repeated letters in any reduced word for~$\pi$. The cycle surplus of~$\pi$ (the excess of permutation cycles over sum components) is at most $\beta(G_\pi)$. Together, these two results sandwich the difference between the Coxeter and reflection lengths:
\[
    \beta(G_\pi) \le \ell(\pi) - \ell'(\pi) \le 2\beta(G_\pi).
\]
The mechanism behind these statements is that every reduced word for~$\pi$ orders the edges of~$G_\pi$ one at a time, with the edges from the first occurrence of each letter forming a spanning forest and the remaining edges coming from the repeated letters.

As a special case of these quantitative statements, the forest case $\beta(G_\pi) = 0$ recovers several classical characterizations of the permutations in $\mathrm{Av}(321, 3412)$, each previously established by different methods. Edelman~\cite{edelman:on-inversions-a:} described these permutations as the direct sums of \emph{unimodal cycles of intervals}: permutations of an interval that can be written as a single cycle whose entries, after cyclic rotation, strictly increase to a maximum and then strictly decrease. Tenner~\cite{tenner:pattern-avoidan:} characterized them as the permutations avoiding the patterns $321$ and $3412$, and equivalently as those whose reduced words contain no repeated letter. Petersen and Tenner~\cite{petersen:the-depth-of-a-:} characterized them as the permutations whose Coxeter and reflection lengths agree. All three descriptions name the same set of permutations, historically called \emph{boolean permutations} because their principal order ideals in the Bruhat order form boolean lattices.

The equivalence between $\beta(G_\pi) = 0$ and $\pi \in \mathrm{Av}(321, 3412)$ follows from a folklore fact: no inversion graph has an induced cycle of length five or more. Since $321$ is the unique permutation whose inversion graph is $C_3$ and $3412$ is the unique permutation whose inversion graph is $C_4$, the graph~$G_\pi$ is a forest if and only if~$\pi$ avoids both patterns. Edelman's description requires more work, and we prove it in Section~\ref{sec:edelman}. Briefly, when~$G_\pi$ is connected and acyclic, every reduced word for~$\pi$ uses each simple reflection exactly once (since $|\mathrm{supp}(\pi)| = n - 1$), and we show that this forces~$\pi$ to be a unimodal cycle of~$[n]$. Taking direct sums recovers Edelman's characterization.

The connected acyclic case also yields a structural refinement: every tree that arises as an inversion graph is a caterpillar. This has been known for some time, and it could also be deduced from Gallai's forbidden induced subgraph characterization of comparability graphs~\cite{gallai:a-translation-o:, gallai:transitiv-orien:}. Brandst\"adt and Kratsch~\cite{brandstadt:on-the-restrict:} obtained it as a byproduct of their characterization of bipartite permutation graphs, Acan and Hitczenko~\cite{acan:on-random-trees:} proved it by induction on the number of vertices while locating the ends of the spine, and it appears again in Gervacio, Rapanut, and Ramos~\cite{gervacio:characterizatio:} and in Brualdi and Dahl~\cite{brualdi:permutation-gra:}. What our proof contributes is a direct derivation from the unimodal-cycle structure of the corresponding permutation.

\begin{theorem*}
For a permutation~$\pi$, the following properties are equivalent:
\begin{enumerate}[label=\textup{(\roman*)}, align=left, leftmargin=*]
    \item $\beta(G_\pi) = 0$, equivalently,~$G_\pi$ is a forest;
    \item no reduced word for~$\pi$ contains a repeated letter;
    \item $\ell(\pi) = \ell'(\pi)$;
    \item $\pi \in \mathrm{Av}(321, 3412)$;
    \item $\pi$ can be expressed as a direct sum of unimodal cycles of intervals;
    \item the Schubert variety $X_\pi$ is toric.
\end{enumerate}
\end{theorem*}

\begin{proof}
The equivalence of (i) and (ii) is Theorem~\ref{thm:repeated-letters-intro}, since $\beta(G_\pi) = 0$ if and only if every reduced word has zero repeated letters. The equivalence of (i) and (iii) is Corollary~\ref{cor:sandwich-intro}. The equivalence of (i) and (iv) is the folklore induced-cycle fact above, for which we provide a short proof as Proposition~\ref{prop:321:3412:forest}. The equivalence of (i) and (v) is the direct-sum form of the tree-versus-unimodal-cycle correspondence proved in Section~\ref{sec:edelman}. Finally, the equivalence of (i) and (vi) is Observation~\ref{obs:t-complexity-beta}, since $X_\pi$ is toric if and only if $c_T(X_\pi) = 0$.
\end{proof}

The cycle-space condition $\beta(G_\pi) = 0$ is at the center of all of these equivalences: each of the others is established by showing it equivalent to (i), and each equivalence is the forest case of one of the motivating quantitative results of this work.


\subsection*{Organization}

In Section~\ref{sec:introduction}, we introduce the permutation-centric objects and properties that we unify in this work. The length of that section reflects the range of results it brings together. Section~\ref{sec:foundations} collects the foundational facts about inversion graphs that we draw on throughout the article: the identification of sum components with connected components of~$G_\pi$, the placement of permutation cycles within sum components, and the absence of long induced cycles in~$G_\pi$. Section~\ref{sec:repeated-letters} develops the reduced-word viewpoint, proves the spanning-forest result that underlies the repeated-letters identity, and discusses when the pattern bound on repeated letters is sharp. Section~\ref{sec:gap-and-forests} proves the surplus bound and combines it with the gap decomposition to sandwich the length gap. Section~\ref{sec:edelman} establishes the correspondence between trees and unimodal cycles, recovers Edelman's description by taking direct sums, and shows that every connected acyclic inversion graph is a caterpillar. We conclude with some perspective on this work, in Section~\ref{sec:conclusion}.

%
%

\section{Permutation context}\label{sec:introduction}


\subsection{Inversion graphs and permutation patterns}

The \emph{inversion graph}~$G_\pi$ of a permutation~$\pi \in S_n$, also known in the literature as the \emph{permutation graph} of~$\pi$, has vertex set $[n] = \{1, 2, \ldots, n\}$, with an edge $ab$ whenever the values~$a$ and $b$ form an inversion of~$\pi$, meaning that the larger value appears to the left of the smaller in one-line notation\footnote{Throughout this paper, permutations written without parentheses are in one-line notation, while permutations written with parentheses and commas are in cycle notation. We compose permutations from right to left, so a simple reflection $s_i = (i, i+1)$ applied on the right of~$\pi$ swaps the entries in positions $i$ and $i+1$ of the one-line notation of~$\pi$, while $s_i$ applied on the left swaps the values $i$ and $i+1$.}. Thus we label the vertices of inversion graphs by \emph{value}, while some other works label them by \emph{position}; the choice is purely cosmetic, but the reader should bear this convention in mind.

The vertex labels exist for convenience of reference (so we can say, for instance, that two specific values \emph{form an inversion}), but inversion graphs themselves are unlabeled in the graph-theoretic sense: we identify two inversion graphs that differ only by a relabeling of their vertices. At times we work more closely with the identification of vertices as values, as in ``the vertex~$1$'' or ``the values of a cycle of~$\pi$''; which sense is intended will always be clear from context. Likewise, in figures we draw the vertex $\pi(i)$ at coordinates $(i, \pi(i))$, so that two vertices are adjacent precisely when one lies southeast of the other; this embedding is merely a convenience for diagrams, and inversion graphs themselves do not come equipped with such an embedding in the plane.

There is another natural way to view~$G_\pi$. A permutation $\pi \in S_n$ defines a poset on~$[n]$, realized by the two linear orders $1 < 2 < \cdots < n$ and $\pi(1) \prec \pi(2) \prec \cdots \prec \pi(n)$, and the inversion graph~$G_\pi$ is the incomparability graph of this poset. This poset has order dimension at most~$2$, and exactly~$2$ unless~$\pi$ is the identity, in which case it is a chain. Dushnik and Miller~\cite{dushnik:partially-order:}, who introduced the dimension of a poset, proved that a poset has dimension at most~$2$ precisely when its incomparability graph is a comparability graph. The corresponding statement for graphs, that a graph arises as an inversion graph if and only if both it and its complement are comparability graphs, is due to Pnueli, Lempel, and Even~\cite{pnueli:transitive-orie:}. Gallai~\cite{gallai:a-translation-o:, gallai:transitiv-orien:} gave a forbidden induced subgraph characterization of comparability graphs.
Inversion graphs are also perfect graphs, and as such they are covered in some depth in Golumbic's text \emph{Algorithmic Graph Theory and Perfect Graphs}~\cite[Chapter~7]{golumbic:algorithmic-gra:}.

The \emph{standardization} of a sequence of $k$ distinct numbers is the permutation of $[k]$ obtained by replacing the smallest entry with $1$, the second smallest with $2$, and so on. Standardization preserves relative order, and hence preserves inversions. A permutation~$\pi$ \emph{contains} a permutation $\sigma$ as a \emph{pattern} if some subsequence of the one-line notation of~$\pi$ standardizes to $\sigma$, and~$\pi$ \emph{avoids} $\sigma$ otherwise. We write $\mathrm{Av}(B)$ for the set of permutations avoiding every pattern in the set $B$. This set is a \emph{permutation class} because it is closed downward under the containment order; we refer to Vatter's survey~\cite{vatter:permutation-cla:} for more background on permutation classes.

By definition, $|E(G_\pi)| = \mathrm{inv}(\pi)$, the number of inversions of~$\pi$. Let $\kappa(G_\pi)$ denote the number of connected components of~$G_\pi$. The \emph{cyclomatic number} of~$G_\pi$, denoted by $\beta(G_\pi)$, is the dimension of the cycle space of~$G_\pi$, equivalently the minimum number of edges that must be removed to obtain a forest, and hence
\[
	\beta(G_\pi) = |E(G_\pi)| - |V(G_\pi)| + \kappa(G_\pi).
\]


\subsection{Reduced words and inversion graphs}
\label{sec:reduced-words}

Multiplication by a simple reflection on either side changes the inversion count by~$1$; in particular, right multiplication on~$\pi$ by~$s_i = (i, i+1)$ swaps the entries in positions~$i$ and~${i+1}$, and creates or destroys the inversion between those two entries. Thus any product expression for~$\pi$ requires at least $\mathrm{inv}(\pi)$ simple reflections, and repeatedly destroying inversions in adjacent positions shows that some expression achieves this bound. We therefore define the \emph{Coxeter length}~$\ell(\pi)$ to be the minimum number of simple reflections in such an expression, and $\ell(\pi) = \mathrm{inv}(\pi) = |E(G_\pi)|$. A \emph{reduced word} for~$\pi$ is an expression of~$\pi$ as a product of $\ell(\pi)$ simple reflections, and we write reduced words as sequences of subscripts: the word $w = i_1 i_2 \cdots i_m$ represents the product $s_{i_1} s_{i_2} \cdots s_{i_m}$, which we build up by right multiplication, one letter at a time.

The \emph{cycle space} of a graph $G$ is the $\mathbb{F}_2$-span of its (graph) cycles, viewing each cycle as a set of edges and using symmetric difference for addition. Its dimension is the cyclomatic number $\beta(G)$: if $G$ has $n$ vertices, then any spanning forest of $G$ has $n - \kappa(G)$ edges, and each of the remaining $\beta(G)$ edges, when added back to the forest, creates a unique cycle. These cycles form a basis for the cycle space.

A reduced word $w = i_1 i_2 \cdots i_m$ for~$\pi$ orders the edges of~$G_\pi$ one at a time. Set $\pi_t = s_{i_1} s_{i_2} \cdots s_{i_t}$, so $\pi_0 = e$ is the identity and $\pi_m = \pi$. At step $t$, the simple reflection $s_{i_t}$ swaps the entries in positions~$i_t$ and~${i_t + 1}$ in the one-line notation of~$\pi_{t-1}$. Because $w$ is reduced, this swap increases the Coxeter length by one, so it creates precisely one new inversion, namely the inversion between the two values in positions $i_t$ and $i_t + 1$; the corresponding edge of~$G_\pi$ joins those two values. Writing $f_t$ for this edge, the $m$ letters of $w$, taken in order, correspond bijectively to the edges $f_1, f_2, \ldots, f_m$ of~$G_\pi$.

For example, take $\pi = 321$, whose inversion graph $G_\pi = K_3$ has the edges $12$, $13$, and $23$. The word $w = 121$ is reduced for~$\pi$, since
\[
    e = 123 \xrightarrow{s_1} 213 \xrightarrow{s_2} 231 \xrightarrow{s_1} 321,
\]
and the edges created in succession are $f_1 = 12$, $f_2 = 13$, and $f_3 = 23$. The letter $1$ appears twice in $w$, so $w$ has one repeated letter, and the cyclomatic number is $\beta(G_{321}) = 3 - 3 + 1 = 1$.

For a more significant example, take $\pi = 243179586$, whose inversion graph is shown on the left in Figure~\ref{fig:spanning-forests-example}. With $|E(G_\pi)| = 10$ and $\kappa(G_\pi) = 2$, we have $\beta(G_\pi) = 3$. The reduced word $w = 1323658768$ has three repeated letters at steps $4, 9, 10$, producing the spanning forest $F(w)$ shown in the middle of Figure~\ref{fig:spanning-forests-example} as solid edges, with the three repeated-letter edges dotted. A different reduced word, $w' = 6781237256$, has its three repeated letters at steps $7, 8, 10$ and produces a different spanning forest, shown on the right.

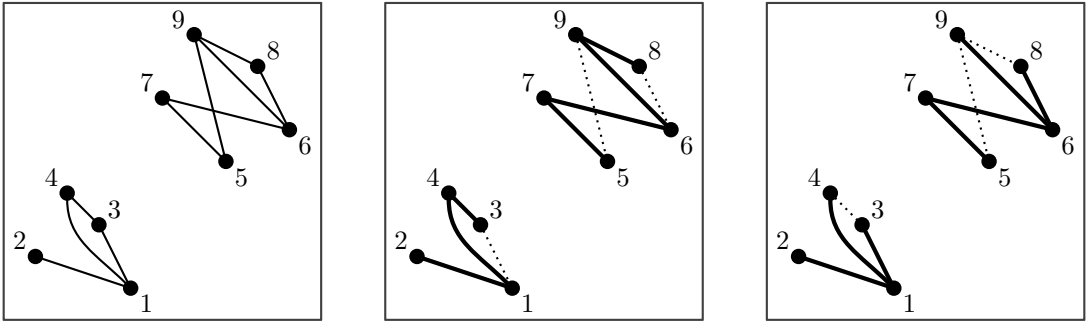
\begin{figure}[htbp]
\begin{center}
\begin{footnotesize}
	\begin{tikzpicture}[scale=0.42,baseline=(current bounding box.center)]
	\draw[darkgray,thick,rounded corners=0.01,line cap=round] (0,0) rectangle (10,10);
	\plotperm{2,4,3,1,7,9,5,8,6}
	\draw (v2)--(v1);
	\draw (v3)--(v1);
	\draw (v4)--(v3);
	\draw [out=-90, in=135] (v4) to (v1);
	\draw (v7)--(v5);
	\draw (v7)--(v6);
	\draw (v8)--(v6);
	\draw (v9)--(v8);
	\draw (v9)--(v5);
	\draw (v9)--(v6);
	\labelnode[-45]{v1}{$1$}
	\labelnode[135]{v2}{$2$}
	\labelnode[45]{v3}{$3$}
	\labelnode[135]{v4}{$4$}
	\labelnode[-45]{v5}{$5$}
	\labelnode[-45]{v6}{$6$}
	\labelnode[135]{v7}{$7$}
	\labelnode[45]{v8}{$8$}
	\labelnode[135]{v9}{$9$}
	\end{tikzpicture}
\qquad
	\begin{tikzpicture}[scale=0.42,baseline=(current bounding box.center)]
	\draw[darkgray,thick,rounded corners=0.01,line cap=round] (0,0) rectangle (10,10);
	\plotperm{2,4,3,1,7,9,5,8,6}
	\draw [ultra thick] (v2)--(v1);
	\draw [dotted] (v3)--(v1);
	\draw [ultra thick] (v4)--(v3);
	\draw [ultra thick, out=-90, in=135] (v4) to (v1);
	\draw [ultra thick] (v7)--(v5);
	\draw [ultra thick] (v7)--(v6);
	\draw [dotted] (v8)--(v6);
	\draw [ultra thick] (v9)--(v8);
	\draw [dotted] (v9)--(v5);
	\draw [ultra thick] (v9)--(v6);
	\labelnode[-45]{v1}{$1$}
	\labelnode[135]{v2}{$2$}
	\labelnode[45]{v3}{$3$}
	\labelnode[135]{v4}{$4$}
	\labelnode[-45]{v5}{$5$}
	\labelnode[-45]{v6}{$6$}
	\labelnode[135]{v7}{$7$}
	\labelnode[45]{v8}{$8$}
	\labelnode[135]{v9}{$9$}
	\end{tikzpicture}
\qquad
	\begin{tikzpicture}[scale=0.42,baseline=(current bounding box.center)]
	\draw[darkgray,thick,rounded corners=0.01,line cap=round] (0,0) rectangle (10,10);
	\plotperm{2,4,3,1,7,9,5,8,6}
	\draw [ultra thick] (v2)--(v1);
	\draw [ultra thick] (v3)--(v1);
	\draw [dotted] (v4)--(v3);
	\draw [ultra thick, out=-90, in=135] (v4) to (v1);
	\draw [ultra thick] (v7)--(v5);
	\draw [ultra thick] (v7)--(v6);
	\draw [ultra thick] (v8)--(v6);
	\draw [dotted] (v9)--(v8);
	\draw [dotted] (v9)--(v5);
	\draw [ultra thick] (v9)--(v6);
	\labelnode[-45]{v1}{$1$}
	\labelnode[135]{v2}{$2$}
	\labelnode[45]{v3}{$3$}
	\labelnode[135]{v4}{$4$}
	\labelnode[-45]{v5}{$5$}
	\labelnode[-45]{v6}{$6$}
	\labelnode[135]{v7}{$7$}
	\labelnode[45]{v8}{$8$}
	\labelnode[135]{v9}{$9$}
	\end{tikzpicture}
\end{footnotesize}
\end{center}
\caption{The inversion graph~$G_\pi$ of $\pi = 243179586$ (left), and two of its spanning forests obtained from reduced words for~$\pi$. The middle forest is $F(w)$ for the reduced word $w = 1323658768$, and the right forest is $F(w')$ for $w' = 6781237256$; in each, the solid edges are the first-occurrence edges and the dotted edges are contributed by repeated letters.}
\label{fig:spanning-forests-example}
\end{figure}

Before the letter $i$ has appeared for the first time (reading left to right), the values in positions $[1, i]$ of the current permutation are precisely $[1, i]$ as nothing has yet crossed that boundary. Then multiplying by the transposition $s_i$ creates an edge joining a value in $[1, i]$ to a value in $[i+1, n]$. We show in Section~\ref{sec:repeated-letters} that the first-occurrence edges form a spanning forest of~$G_\pi$, and so the remaining edges, contributed by repeated letters, account for the $\beta(G_\pi)$ edges that lie outside this spanning forest.

Let $\mathrm{supp}(\pi)$ denote the \emph{support} of~$\pi$, the set of simple reflections appearing in any reduced word for~$\pi$. This set is independent of the choice of reduced word, since by Matsumoto's theorem~\cite{matsumoto:generateurs-et-:} any two reduced words are connected by braid and commutation relations, both of which preserve the set of generators used. Then $|\mathrm{supp}(\pi)|$ is the number of distinct letters in any reduced word, and the number of repeated letters in $w$, meaning the number of occurrences past the first occurrence of each distinct letter, is $\ell(\pi) - |\mathrm{supp}(\pi)|$. The spanning-forest description identifies this count with the cyclomatic number.

\begin{theorem}
\label{thm:repeated-letters-intro}
Every reduced word for a permutation~$\pi$ contains exactly $\beta(G_\pi)$ repeated letters.
\end{theorem}

The numerical identity in Theorem~\ref{thm:repeated-letters-intro} also follows by combining $\ell(\pi)=|E(G_\pi)|$ with the standard identity $|\mathrm{supp}(\pi)|=n-\kappa(G_\pi)$. What the theorem contributes is the structural mechanism behind this equality: every reduced word orders the edges of $G_\pi$ so that first occurrences form a spanning forest, while repeated letters are precisely the edges outside that forest.

Tenner~\cite{tenner:pattern-avoidan:} showed that this repetition count is zero precisely for the permutations avoiding $321$ and $3412$, and Daly~\cite{daly:reduced-decompo:} characterized the permutations whose reduced words contain exactly one repeated letter. Tenner~\cite{tenner:repetition-in-r:} later proved a general upper bound: the number of repeated letters in any reduced word for~$\pi$ is at most the total number of copies of $321$ and $3412$ in~$\pi$. Theorem~\ref{thm:repeated-letters-intro} replaces this bound with the exact value $\beta(G_\pi)$. Tenner's bound then follows from two graph-theoretic facts: induced cycles span the cycle space, so their number is at least $\beta(G_\pi)$, and inversion graphs have no induced cycles of length five or more (folklore, presented here as Proposition~\ref{prop:no-long-induced-cycles}), so every induced cycle is a triangle or a $4$-cycle, corresponding to a copy of $321$ or $3412$ in~$\pi$.

\begin{corollary}[cf.~Tenner~{\cite[Theorem~3.2]{tenner:repetition-in-r:}}]
\label{cor:pattern-bound-intro}
The number of repeated letters in any reduced word for~$\pi$ is at most the total number of copies of $321$ and $3412$ in~$\pi$.
\end{corollary}


\subsection{Permutation cycles and the length gap}

The cycle space of~$G_\pi$ is also closely tied to the cycle structure of~$\pi$, via the sum components of~$\pi$ (the indecomposable summands in its direct sum decomposition; see Section~\ref{sec:foundations} for a formal definition). Writing $\kappa(\pi)$ for the number of sum components of~$\pi$, the first fact we need, established in Proposition~\ref{prop:sum-components}, is that the sum components of~$\pi$ correspond to the connected components of~$G_\pi$, so $\kappa(\pi) = \kappa(G_\pi)$.

Every cycle of~$\pi$ is contained in a single sum component (Proposition~\ref{prop:cycles-in-components}). Let $c(\pi)$ denote the number of cycles of~$\pi$ (including fixed points). Thus, $c(\pi) \ge \kappa(\pi)$. We call the nonnegative difference $c(\pi) - \kappa(\pi)$ the \emph{cycle surplus} of~$\pi$, since it counts the excess of permutation cycles over sum components.

For example, the permutation $321$ has cycle structure $(1, 3)(2)$, so $c(321) = 2$, and its inversion graph $G_{321} = K_3$ is connected, so $\kappa(321) = 1$ and the cycle surplus is $1$ and $\beta(G_{321}) = 1$ as computed above, so the cycle surplus and cyclomatic number agree here.

Our larger example $\pi = 243179586 = (1, 2, 4)(3)(5, 7)(6, 9)(8)$ from Figure~\ref{fig:spanning-forests-example} has $c(\pi) = 5$, and $G_\pi$ has $\kappa(\pi) = 2$ connected components, giving a cycle surplus of $3$, which again equals $\beta(G_\pi) = 3$.

These two quantities need not agree in general. Take $\pi = 4321 = (1, 4)(2, 3)$, so $c(4321) = 2$, and the inversion graph $G_{4321} = K_4$ is connected, so $\kappa(4321) = 1$ and the cycle surplus is again $1$. But $|E(G_{4321})| = 6$, so $\beta(G_{4321}) = 6 - 4 + 1 = 3$, strictly larger than the cycle surplus.

\begin{theorem}
\label{thm:surplus-bound-intro}
For every permutation~$\pi$,
\[
    c(\pi) - \kappa(\pi) \le \beta(G_\pi).
\]
\end{theorem}

The proof, in Section~\ref{sec:gap-and-forests}, rests on a parity argument: between any two disjoint~$\pi$-invariant subsets, the number of edges of~$G_\pi$ joining them is even. In particular, the supports of any two cycles of~$\pi$ are joined by an even number of edges, and contracting each support to a single vertex turns this evenness into the bound on the cyclomatic number.

The \emph{reflection length} $\ell'(\pi)$ of a permutation $\pi \in S_n$ is the minimum number of reflections, meaning arbitrary transpositions $(i, j)$ rather than only the simple transpositions, needed to express~$\pi$. Decomposing~$\pi$ into its $c(\pi)$ disjoint cycles and expressing each $m$-cycle as a product of $m - 1$ transpositions gives an expression for~$\pi$ using $n - c(\pi)$ transpositions, so $\ell'(\pi) \le n - c(\pi)$. The reverse inequality holds because multiplying by a transposition changes the cycle count by $\pm 1$: starting from the identity, which has~$n$ cycles, any expression of~$\pi$ as a product of transpositions must use at least $n - c(\pi)$ transpositions to reach~$c(\pi)$ cycles. Combining $\ell(\pi) = |E(G_\pi)|$ with $\ell'(\pi) = n - c(\pi)$, and adding and subtracting $\kappa(\pi)$, gives the following ``gap decomposition''.

\begin{theorem}
\label{thm:gap-decomposition-intro}
For every permutation~$\pi$,
\[
    \ell(\pi) - \ell'(\pi) = \beta(G_\pi) + \bigl(c(\pi) - \kappa(\pi)\bigr).
\]
\end{theorem}

\begin{proof}
Let $\pi \in S_n$. We have
\begin{align*}
    \ell(\pi) - \ell'(\pi)
    &= |E(G_\pi)| - n + c(\pi) \\
    &= \bigl(|E(G_\pi)| - n + \kappa(\pi)\bigr) + \bigl(c(\pi) - \kappa(\pi)\bigr) \\
    &= \beta(G_\pi) + \bigl(c(\pi) - \kappa(\pi)\bigr). \qedhere
\end{align*}
\end{proof}

Both summands in the gap decomposition of Theorem~\ref{thm:gap-decomposition-intro} are nonnegative, and by Theorem~\ref{thm:surplus-bound-intro} the second is at most the first. Together these sandwich the length gap.

\begin{corollary}
\label{cor:sandwich-intro}
For every permutation~$\pi$,
\[
    \beta(G_\pi) \le \ell(\pi) - \ell'(\pi) \le 2\beta(G_\pi).
\]
\end{corollary}

In particular, $\ell(\pi) = \ell'(\pi)$ if and only if~$G_\pi$ is a forest.

Both lengths count transpositions in an expression for~$\pi$, so the sign of~$\pi$ is both $(-1)^{\ell(\pi)}$ and $(-1)^{\ell'(\pi)}$, and hence $\ell(\pi) - \ell'(\pi)$ is even. Theorem~\ref{thm:gap-decomposition-intro} then gives $\beta(G_\pi) \equiv c(\pi) - \kappa(\pi) \pmod 2$, so the lower bound above can be attained only when $\beta(G_\pi)$ is even.

Both bounds are attained. The upper bound is attained by $321$, where we saw above that $\beta(G_{321}) = 1$ and that the cycle surplus is~$1$, so $\ell(321) - \ell'(321) = 2$. For the lower bound, parity rules out an example with cyclomatic number~$1$, but the next value does occur: the permutation $3421$ is a single $4$-cycle with five inversions, so $\beta(G_{3421}) = 5 - 4 + 1 = 2$ and $\ell(3421) - \ell'(3421) = 5 - 3 = 2$.


\subsection{Schubert varieties}

The cyclomatic number $\beta(G_\pi)$ also appears in the Schubert geometry of the flag variety. To each permutation $\pi \in S_n$ is associated a \emph{Schubert variety} $X_\pi$, a classical algebraic variety on which a torus~$T$ acts; the codimension of a generic $T$-orbit is called the \emph{$T$-complexity} of $X_\pi$, denoted by~$c_T(X_\pi)$. We do not develop this setting here, but we record what known results say about it from our viewpoint.

In type $A$, this complexity is $\ell(\pi) - |\mathrm{supp}(\pi)|$; see Lee, Masuda, and Park~\cite{lee:toric-bruhat-in:} and Donten-Bury, Escobar, and Portakal~\cite[Theorem~5.8 and Corollary~4.16]{donten-bury:complexity-of-t:}, and Gao and Hodges~\cite{gao:orbit-structure:fpsac} for a type-uniform treatment.

\begin{observation}
\label{obs:t-complexity-beta}
For every permutation $\pi \in S_n$, we have $c_T(X_\pi) = \beta(G_\pi)$.
\end{observation}

The forest case, $\beta(G_\pi) = 0$, is the case where $X_\pi$ is a \emph{toric variety} (a variety on which $T$ acts with a dense orbit); this special case was established earlier by Karuppuchamy~\cite{karuppuchamy:on-schubert-var:}.

Neuhaus, Portakal, and Paul~\cite{neuhaus:torus-actions-o:} express the complexity of a Kazhdan--Lusztig variety as the cyclomatic number of a graph attached to a chain in a Bruhat interval, and show that the variety is toric precisely when that graph is a forest. Their graph is built from the transpositions along the chain rather than from inversions, so this is a different statement from Observation~\ref{obs:t-complexity-beta}.

%
%

\section{Sum components and induced cycles}
\label{sec:foundations}

This section collects the structural facts about inversion graphs that we draw on throughout the paper. We first set up the direct-sum decomposition of a permutation and identify the sum components of~$\pi$ with the connected components of~$G_\pi$. Then we show that inversion graphs have no induced cycles of length five or more, and identify their induced triangles and induced $4$-cycles with copies of $321$ and $3412$.

\subsection{Sum components and connected components}
\label{subsec:sum-components}

The \emph{direct sum} of permutations $\pi_1 \in S_m$ and $\pi_2 \in S_n$ is the permutation $\pi_1 \oplus \pi_2 \in S_{m + n}$ whose one-line notation is obtained by concatenating the one-line notation of~$\pi_1$ with the one-line notation of~$\pi_2$, with each entry of the latter increased by~$m$. A permutation is \emph{sum indecomposable} if it cannot be written as the direct sum of two nonempty permutations; otherwise it can be written uniquely as $\alpha_1 \oplus \cdots \oplus \alpha_k$ with each $\alpha_i$ sum indecomposable, and these summands are the \emph{sum components} of~$\pi$.

The foundational fact of this section identifies these sum components with the connected components of~$G_\pi$. The result is folklore; we include a proof for completeness.

\begin{proposition}
\label{prop:sum-components}
The connected components of~$G_\pi$ are the sets of values of the sum components of~$\pi$. In particular, a permutation $\pi$ is sum indecomposable if and only if~$G_\pi$ is connected.
\end{proposition}

\begin{proof}
Suppose $a < b < c$ and $ac$ is an edge of~$G_\pi$. Then $c$ appears to the left of~$a$ in the one-line notation of~$\pi$. If $b$ appears to the right of~$a$, then $c$ appears to the left of $b$, so $cb$ is an edge. If $b$ appears to the left of~$a$, then $ba$ is an edge. Thus every value between the endpoints of an edge lies in the same connected component as that edge, and so every connected component of $G_\pi$ is an interval of values.

Let $I = [a, c]$ be a connected component of~$G_\pi$. Since no edge joins $I$ to its complement, every value less than~$a$ appears to the left of every value in $I$, and every value greater than $c$ appears to the right of every value in $I$. Hence the values in $I$ occupy precisely the positions $a, a+1, \ldots, c$.

It follows that~$\pi$ is the direct sum of its restrictions to the connected components of~$G_\pi$, ordered by value. Each of these restrictions is sum indecomposable, since a further direct-sum decomposition would disconnect the corresponding graph component. Thus these restrictions are precisely the sum components.
\end{proof}

We now turn to the basic interaction between the cycle structure of~$\pi$ and its sum decomposition. For a subset $I \subseteq [n]$, write $G_\pi[I]$ for the subgraph of~$G_\pi$ induced on $I$.

\begin{proposition}
\label{prop:cycles-in-components}
Every cycle of~$\pi$ is contained in a single sum component, and if $I \subseteq [n]$ is the set of values of a cycle of~$\pi$, then $G_\pi[I]$ is connected.
\end{proposition}

\begin{proof}
Each sum component is~$\pi$-invariant: the values in a sum component occupy a contiguous interval of positions, and~$\pi$ maps this interval to itself. A cycle of~$\pi$ therefore stays within a single sum component.

For the second claim, restrict~$\pi$ to $I$ and standardize. The resulting permutation is a single cycle, hence sum indecomposable, so by Proposition~\ref{prop:sum-components} its inversion graph is connected. Standardization preserves inversions, so this graph is isomorphic to $G_\pi[I]$.
\end{proof}

In particular, $c(\pi) \ge \kappa(\pi)$, so the cycle surplus $c(\pi) - \kappa(\pi)$ is indeed nonnegative.

By Proposition~\ref{prop:sum-components}, $\kappa(\pi) = \kappa(G_\pi)$. We use the two notations interchangeably from this point on.

\subsection{Induced cycles of inversion graphs}
\label{subsec:induced-cycles}

Inversion graphs cannot contain arbitrary induced cycles: their induced cycles are restricted to lengths three and four. This fact is folklore; we include a proof for completeness.

\begin{proposition}
\label{prop:no-long-induced-cycles}
No inversion graph contains an induced cycle of length five or more.
\end{proposition}

\begin{proof}
We prove an equivalent statement: any induced cycle of length at least four in an inversion graph has length exactly four.

Suppose~$G_\pi$ contains an induced cycle on vertices $\pi(i_1), \ldots, \pi(i_k)$ with $k \ge 4$. Choose $i_1$ minimal among the indices of the cycle. The vertex $\pi(i_1)$ has two neighbors in the cycle; let $i_2$ be the index of the one whose position is closer to $i_1$, and label the remaining cycle vertices $\pi(i_3), \ldots, \pi(i_k)$ in cycle order, so that
\[
    \pi(i_1)\pi(i_2)\,, \ \pi(i_2)\pi(i_3)\,, \ \dots\,, \ \pi(i_{k-1})\pi(i_k)\,, \ \pi(i_k)\pi(i_1)
\]
are edges of~$G_\pi$. By minimality of $i_1$ and choice of $i_2$, we have $i_1 < i_2 < i_k$, and the point $(i_2, \pi(i_2))$ lies southeast of $(i_1, \pi(i_1))$. Since $\pi(i_3)\pi(i_2)$ is an edge and $\pi(i_3)\pi(i_1)$ is not, using that $k \ge 4$ and the cycle is induced, we have $i_1 < i_3 < i_2$ and $\pi(i_2) < \pi(i_1) < \pi(i_3)$.

Now consider $(i_k, \pi(i_k))$. Since $i_k > i_2$ and $\pi(i_k)\pi(i_1)$ is an edge, the point $(i_k, \pi(i_k))$ lies east of $(i_2, \pi(i_2))$ and southeast of $(i_1, \pi(i_1))$. These constraints force $(i_k, \pi(i_k))$ to lie southeast of $(i_3, \pi(i_3))$, giving an edge $\pi(i_k)\pi(i_3)$. Since the cycle is induced, this forces $k = 4$.
\end{proof}

The induced cycles of~$G_\pi$ are in bijection with copies of two small patterns.

\begin{proposition}
\label{prop:321:3412:forest}
The inversion graph~$G_\pi$ is a forest if and only if~$\pi$ avoids the patterns $321$ and~$3412$. More precisely, induced triangles of~$G_\pi$ are in bijection with copies of $321$ in~$\pi$, and induced $4$-cycles of~$G_\pi$ are in bijection with copies of $3412$ in~$\pi$.
\end{proposition}

\begin{proof}
A triangle of~$G_\pi$ requires each pair of the corresponding entries to be an inversion. The only way this is possible is if the three entries form a copy of $321$. The converse direction is immediate because every copy of $321$ in $\pi$ corresponds to an induced triangle in $G_\pi$.

A $4$-cycle is a copy of~$K_{2,2}$, and for $G_\pi$ to contain an induced copy of $K_{2,2}$, the two bipartition classes would have to correspond to noninversions of~$\pi$, while every pair of entries with one endpoint in each class would have to correspond to an inversion. Thus the two larger entries must appear first, in increasing order, followed by the two smaller entries, also in increasing order; that is, the four entries must form a copy of $3412$. Conversely, every copy of $3412$ in~$\pi$ corresponds to an induced $4$-cycle in~$G_\pi$.

By Proposition~\ref{prop:no-long-induced-cycles}, induced cycles of~$G_\pi$ have length three or four, and a graph containing a cycle contains an induced one (for instance a shortest cycle). Thus,~$G_\pi$ is a forest if and only if it contains no induced triangle and no induced $4$-cycle, which by the previous paragraphs is if and only if~$\pi$ avoids $321$ and $3412$.
\end{proof}

%
%

\section{Reduced words and inversion graphs}
\label{sec:repeated-letters}

This section proves Theorem~\ref{thm:repeated-letters-intro}, the identity $\ell(\pi) - |\mathrm{supp}(\pi)| = \beta(G_\pi)$, and derives the pattern-count result, Corollary~\ref{cor:pattern-bound-intro}, as a consequence. Every reduced word for~$\pi$ orders the edges of~$G_\pi$ one at a time, and we show that the edges contributed by first occurrences of letters form a spanning forest of~$G_\pi$. The repeated letters then contribute precisely the edges outside this forest, of which there are~$\beta(G_\pi)$.

As in Section~\ref{sec:reduced-words}, let $w = i_1 i_2 \cdots i_m$ be a reduced word for~$\pi$, and set $\pi_t = s_{i_1} \cdots s_{i_t}$, so $\pi_0 = e$ and $\pi_m = \pi$. At step $t$, applying $s_{i_t}$ swaps the entries in positions $i_t$ and $i_t + 1$ of $\pi_{t-1}$. Let $a_t < b_t$ be the values in those positions just before the swap, so that the swap creates the inversion $(a_t, b_t)$ of $\pi_t$. We write $f_t = a_t b_t$ for the corresponding edge of~$G_\pi$.

\begin{proposition}
\label{prop:reduced-word-adds-edge}
The edges $f_1, f_2, \ldots, f_m$ are distinct and account for every edge of~$G_\pi$.
\end{proposition}

\begin{proof}
Since $w$ is reduced, the inversion set of $\pi_t$ has size $t$ for each $t$. The swap at step $t$ adds the inversion $(a_t, b_t)$, which was not previously an inversion, so the inversion set of $\pi_t$ is the inversion set of $\pi_{t-1}$ together with the new inversion $f_t$. Iterating, the inversions $f_1, \ldots, f_m$ are distinct and constitute the entire inversion set of~$\pi$, and the edges of~$G_\pi$ are precisely the inversions of~$\pi$.
\end{proof}

\subsection{First occurrences and spanning forests}
\label{subsec:first-occurrences}

Fix a permutation $\pi$ and let $w = i_1i_2\cdots i_m$ be a reduced word for $\pi$. Say that the $t$th letter of $w$ is a \emph{first occurrence} if $i_t \notin \{i_1, \ldots, i_{t-1}\}$, and let $F(w) \subseteq E(G_\pi)$ be the set of edges created by first occurrences in~$w$.

The next proposition is the core structural claim. It says that $F(w)$ is always a spanning forest of~$G_\pi$, regardless of which reduced word is chosen. The example of $\pi = 243179586$ in Figure~\ref{fig:spanning-forests-example} illustrates the construction concretely. The two reduced words $w$ and $w'$ given there produce different spanning forests $F(w)$ and $F(w')$, but each has $n - \kappa(G_\pi) = 7$ edges, with three on the connected component $\{1, 2, 3, 4\}$ and four on the connected component $\{5, 6, 7, 8, 9\}$. The proposition says this is no accident.

The proof of Proposition~\ref{prop:spanning-forest} tracks two pieces of information across the steps of $w$. The first is a partition of~$[n]$ into intervals, which coarsens as new letters appear: after step $t$, two values $x < y$ lie in the same interval if and only if each of the letters $x, x + 1, \ldots, y - 1$ has already appeared in $w$. The second is the location of the values themselves: after step $t$, the values in any interval $I$ occupy precisely the positions of $I$ in the current permutation. They have been shuffled by the reflections applied so far, but have not crossed into a neighboring interval.

\begin{proposition}
\label{prop:spanning-forest}
For every reduced word $w$ of a permutation~$\pi$, the set $F(w)$ is a spanning forest of~$G_\pi$.
\end{proposition}

\begin{proof}
For $t \ge 0$, let $A_t \subseteq [n-1]$ be the set of letters appearing among $i_1, \ldots, i_t$, and let $\mathcal{P}_t$ be the partition of~$[n]$ into the maximal intervals obtained by cutting between $i$ and $i+1$ for each $i \in [n-1] \setminus A_t$. That is, if $[n-1] \setminus A_t = \{j_1 < j_2 < \cdots < j_q\}$, then the parts of $\mathcal{P}_t$ are the intervals $[j_k + 1, j_{k+1}]$ for $0 \le k \le q$, where $j_0 = 0$ and $j_{q+1} = n$. Let $F_t = \{f_s : s \le t \text{ and } i_s \text{ is a first occurrence}\}$, so that $F_m = F(w)$. We prove, by induction on $t$, that:
\begin{enumerate}[label=(\roman*)]
    \item
    \label{prop:spanning-forest:i}
    $F_t$ is an acyclic subgraph of~$G_\pi$ whose connected components are precisely the parts of $\mathcal{P}_t$;
    \item
    \label{prop:spanning-forest:ii}
    for each part $B$ of $\mathcal{P}_t$, the values occupying the positions of $B$ in the one-line notation of $\pi_t$ are precisely the values in $B$.
\end{enumerate}
Both invariants hold at $t = 0$: we have $A_0 = \emptyset$, the partition $\mathcal{P}_0$ consists of singletons, $F_0 = \emptyset$, and $\pi_0 = e$.

Suppose the invariants hold at time $t - 1$. If $i_t \in A_{t-1}$, then $A_t = A_{t-1}$, $\mathcal{P}_t = \mathcal{P}_{t-1}$, and $F_t = F_{t-1}$. Positions $i_t$ and $i_t + 1$ lie in the same part $B$ of $\mathcal{P}_{t-1}$, since no cut of $\mathcal{P}_{t-1}$ separates them. By invariant~\ref{prop:spanning-forest:ii}, the values at these two positions are both in $B$, so applying $s_{i_t}$ swaps two values of $B$ within positions of $B$. Both invariants are preserved, and we note for later use that both endpoints of the edge $f_t$ lie in $B$.

If instead $i_t$ is a first occurrence, then $\mathcal{P}_t$ is obtained from $\mathcal{P}_{t-1}$ by merging the two parts $B_L$ and $B_R$ on either side of the cut at $i_t$. Positions $i_t$ and $i_t + 1$ lie in $B_L$ and $B_R$ respectively, and by invariant~\ref{prop:spanning-forest:ii} the values at those positions are in $B_L$ and $B_R$; so the edge $f_t$ joins $B_L$ to $B_R$. As $B_L$ and $B_R$ are distinct components of $F_{t-1}$, the set $F_t = F_{t-1} \cup \{f_t\}$ is acyclic and has $B_L \cup B_R$ as a component, in agreement with $\mathcal{P}_t$; this preserves invariant~\ref{prop:spanning-forest:i}. The swap exchanges two values between the positions of $B_L$ and those of $B_R$, so the values of $B_L \cup B_R$ still occupy the positions of $B_L \cup B_R$ in $\pi_t$; this preserves invariant~\ref{prop:spanning-forest:ii}. In this case both endpoints of $f_t$ lie in the part $B_L \cup B_R$ of $\mathcal{P}_t$.

In either case, the endpoints of $f_t$ lie in a single part of $\mathcal{P}_t$. Since the partitions only coarsen as $t$ increases, every edge of~$G_\pi$ has both endpoints in a single part of the final partition $\mathcal{P}_m$. On the other hand, invariant~\ref{prop:spanning-forest:i} shows that each part of $\mathcal{P}_m$ is connected in~$G_\pi$, indeed already connected using only the edges of $F_m$. The parts of $\mathcal{P}_m$ are therefore precisely the connected components of~$G_\pi$, and $F(w) = F_m$ is an acyclic subgraph of~$G_\pi$ whose components are these parts, hence a spanning forest of~$G_\pi$.
\end{proof}

Proposition~\ref{prop:spanning-forest} also gives the identity $|\mathrm{supp}(\pi)| = n - \kappa(G_\pi)$: the forest $F(w)$ has one edge per first-occurrence letter, hence $|\mathrm{supp}(\pi)|$ edges in all, and every spanning forest of~$G_\pi$ has $n - \kappa(G_\pi)$ edges.

\subsection{Repeated letters and the cyclomatic number}
\label{subsec:repeated-letters-cyclomatic}

The spanning-forest description shows that the number of repeated letters in any reduced word is a graph invariant. A reduced word $w$ for~$\pi$ has $\ell(\pi) = |E(G_\pi)|$ letters, one per edge of~$G_\pi$ by Proposition~\ref{prop:reduced-word-adds-edge}. The first-occurrence letters contribute the $n - \kappa(G_\pi)$ edges of the spanning forest $F(w)$, by Proposition~\ref{prop:spanning-forest}. The repeated letters therefore account for the remaining edges, giving the following theorem.

\begin{theorem}
\label{thm:repeated-letters}
Every reduced word for a permutation~$\pi$ contains exactly $\beta(G_\pi)$ repeated letters.
\end{theorem}

The permutations with no repetition in any reduced word are easy to identify. By Theorem~\ref{thm:repeated-letters}, no reduced word for~$\pi$ contains a repeated letter precisely when $\beta(G_\pi) = 0$, that is, when~$G_\pi$ is a forest.

\begin{corollary}
\label{cor:tenner-no-repeats}
No reduced word for~$\pi$ contains a repeated letter if and only if~$G_\pi$ is a forest.
\end{corollary}

Combined with the avoidance characterization of forests given by Proposition~\ref{prop:321:3412:forest}, Corollary~\ref{cor:tenner-no-repeats} recovers Tenner's theorem~\cite{tenner:pattern-avoidan:} that the permutations whose reduced words contain no repeated letter are precisely those avoiding $321$ and $3412$.

The number of repeated letters can also be bounded above by pattern counts. The following bound was established by Tenner~\cite{tenner:repetition-in-r:}; Theorem~\ref{thm:repeated-letters} gives a short graph-theoretic proof.

\begin{corollary}
\label{cor:repetitions-by-patterns}
The number of repeated letters in any reduced word for~$\pi$ is at most the total number of copies of $321$ and $3412$ in~$\pi$.
\end{corollary}

\begin{proof}
The induced cycles of a graph span its cycle space: any cycle with a chord is the symmetric difference of two shorter cycles, so induction on length expresses every cycle as a sum of chordless cycles. The number of induced cycles is therefore at least the cyclomatic number. By Proposition~\ref{prop:no-long-induced-cycles}, every induced cycle of~$G_\pi$ has length $3$ or $4$, and by the proof of Proposition~\ref{prop:321:3412:forest} these correspond to copies of $321$ and $3412$ in~$\pi$, respectively. The corollary follows from Theorem~\ref{thm:repeated-letters}.
\end{proof}

Let $[321, 3412](\pi)$ denote the total number of copies of $321$ and $3412$ in~$\pi$. We have just seen, by combining Proposition~\ref{prop:321:3412:forest} with Corollary~\ref{cor:tenner-no-repeats}, that when $[321, 3412](\pi) = 0$, every reduced word for~$\pi$ has zero repeated letters. Daly~\cite{daly:reduced-decompo:} showed the analogous statement for the next value: when $[321, 3412](\pi) = 1$, every reduced word for~$\pi$ has exactly one repeated letter. The cycle-space viewpoint extends this one step further. If $[321, 3412](\pi) \le 2$, then~$G_\pi$ has at most two induced cycles, and since two distinct cycles are distinct nonzero elements of the cycle space, they are linearly independent over $\mathbb{F}_2$; thus $\beta(G_\pi) = [321, 3412](\pi)$ and every reduced word for~$\pi$ contains $[321, 3412](\pi)$ repeated letters. The conclusion fails for the next value: the permutation $\pi = 34512$ has inversion graph $K_{3, 2}$, with three induced $4$-cycles and no triangles, so $[321, 3412](34512) = 3$ while $\beta(G_{34512}) = 2$.

Corollary~\ref{cor:repetitions-by-patterns} is sharp precisely when the induced cycles of~$G_\pi$ form a basis of the cycle space. By Theorem~\ref{thm:repeated-letters} and Proposition~\ref{prop:321:3412:forest}, both sides of the bound depend only on $G_\pi$ (the repetition count is $\beta(G_\pi)$; the total $321$ and $3412$ count is the number of induced triangles and induced $4$-cycles), so sharpness is an invariant of the inversion graph. Tenner~\cite{tenner:repetition-in-r:} characterized the permutations for which the bound is sharp as those avoiding ten particular patterns.
The inversion graphs of these ten patterns fall into exactly four isomorphism classes, and grouping the patterns by their inversion graphs gives
\[
	\begin{array}{ll}
		K_4: & 4321,\\[4pt]
		K_{3,2}: & 34512, 45123,\\[4pt]
		\overline{K}_2 \vee P_3: & 45231, 45312, 53412,\\[4pt]
		\overline{K}_2 \vee (K_1 \cup K_2): & 35412, 43512, 45132, 45213,
	\end{array}
\]
where $G \vee H$ denotes the \emph{join}, the graph obtained from disjoint copies of $G$ and $H$ by adding every edge between them. These ten permutations are all of the permutations realizing the four graphs, so containing one of the patterns is the same as having one of the graphs as an induced subgraph, and Tenner's theorem takes a purely graph-theoretic form: the induced triangles and induced $4$-cycles of an inversion graph are linearly independent in its cycle space if and only if none of the four occurs in it as an induced subgraph. None of the four contains another, so they are the minimal obstructions to independence. This graph-theoretic statement can also be proved on its own terms, without the pattern classification, but we do not give that argument here.

%
%

\section{The length gap}
\label{sec:gap-and-forests}

Theorem~\ref{thm:gap-decomposition-intro} establishes the gap decomposition $\ell(\pi) - \ell'(\pi) = \beta(G_\pi) + \bigl(c(\pi) - \kappa(\pi)\bigr)$ by substitution, and Section~\ref{sec:foundations} shows that every cycle of~$\pi$ is contained in a single sum component, with the values of that cycle inducing a connected subgraph of~$G_\pi$ (Proposition~\ref{prop:cycles-in-components}). What remains is the conceptually nontrivial part: that the cycle surplus is bounded above by the cyclomatic number,
\[
    c(\pi) - \kappa(\pi) \le \beta(G_\pi).
\]
This says that the cycle space of~$G_\pi$ controls how many more cycles than sum components~$\pi$ can have. Combined with the gap decomposition, it sandwiches the length gap as $\beta(G_\pi) \le \ell(\pi) - \ell'(\pi) \le 2 \beta(G_\pi)$ and forces the gap to vanish precisely when~$G_\pi$ is a forest.

The proof rests on the following parity argument.

\begin{proposition}
\label{prop:cross-edges-even}
Let $I$ and $J$ be disjoint~$\pi$-invariant subsets of~$[n]$. The number of edges of~$G_\pi$ with one endpoint in $I$ and the other in $J$ is even.
\end{proposition}

\begin{proof}
Let $\sigma$ denote the restriction of~$\pi$ to $I \cup J$, a permutation of the set $I \cup J$; its inversions are the inversions of~$\pi$ with both endpoints in $I \cup J$, and its sign is determined by their parity. Because $I$ and $J$ are~$\pi$-invariant, $\sigma$ is the product of two disjoint permutations supported on $I$ and on $J$ respectively, so
\[
    \mathrm{sgn}(\sigma) = \mathrm{sgn}(\sigma|_I) \cdot \mathrm{sgn}(\sigma|_J).
\]
The inversions of $\sigma$ split into three groups by where their endpoints lie: both within~$I$, both within~$J$, or one in each. The first two groups determine the signs $\mathrm{sgn}(\sigma|_I)$ and $\mathrm{sgn}(\sigma|_J)$, so the sign identity forces the third group to be even, and its inversions are precisely the edges of~$G_\pi$ between~$I$ and~$J$.
\end{proof}

We can now bound the cycle surplus above by the cyclomatic number.

\begin{theorem}
\label{thm:surplus-bound}
For every permutation~$\pi$,
\[
    c(\pi) - \kappa(\pi) \le \beta(G_\pi).
\]
In particular, if~$G_\pi$ is a forest, then $c(\pi) = \kappa(\pi)$.
\end{theorem}

\begin{proof}
It suffices to verify the inequality inside one connected component $C$ of~$G_\pi$, since both quantities are additive over connected components. Let $I_1, \ldots, I_r$ be the value sets of the cycles of~$\pi$ contained in the corresponding sum component, with $r \ge 1$. By Proposition~\ref{prop:cycles-in-components}, each $G_\pi[I_j]$ is connected, and by Proposition~\ref{prop:cross-edges-even}, the number of edges of~$G_\pi$ between any two of the $I_j$ is even.

The contribution of $C$ to the cycle surplus is $r - 1$. We show that its contribution to~$\beta(G_\pi)$ is at least $r - 1$. Contract each $G_\pi[I_j]$ in $C$ to a single vertex, discarding the resulting loops but retaining parallel edges between distinct contracted vertices, and let $Q$ denote the resulting loopless multigraph on $r$ vertices. Since $C$ is connected, so is $Q$, and hence at least $r - 1$ pairs of its vertices are joined by an edge. By Proposition~\ref{prop:cross-edges-even}, every such pair is joined by an even number of edges, hence by at least two, so $|E(Q)| \ge 2(r - 1)$. Counting edges,
\[
    \beta(C) = |E(C)| - |V(C)| + 1 = \sum_{j=1}^r \beta(G_\pi[I_j]) + |E(Q)| - r + 1 \ge 2(r - 1) - r + 1 = r - 1,
\]
which is the desired bound on the contribution of $C$.

Summing over all connected components gives $c(\pi) - \kappa(\pi) \le \beta(G_\pi)$.
\end{proof}

We can now collect the consequences. The gap decomposition is proved in Section~\ref{sec:introduction} by direct substitution; combined with the surplus bound, it yields a characterization of permutations with $\ell(\pi) = \ell'(\pi)$.

\begin{corollary}
\label{cor:forest-characterization}
The equality $\ell(\pi) = \ell'(\pi)$ holds if and only if~$G_\pi$ is a forest, in which case~$\pi$ has exactly $\kappa(\pi)$ cycles, one supported on each sum component.
\end{corollary}

\begin{proof}
The gap decomposition has two nonnegative terms. If the gap vanishes, both terms must be zero, so $\beta(G_\pi) = 0$; that is,~$G_\pi$ is a forest. Conversely, if~$G_\pi$ is a forest, then $\beta(G_\pi) = 0$, and Theorem~\ref{thm:surplus-bound} forces $c(\pi) = \kappa(\pi)$. The cycle structure in the forest case follows from Proposition~\ref{prop:cycles-in-components}.
\end{proof}

These results have a particularly clean structural consequence: when~$G_\pi$ is a forest, the sum decomposition of~$\pi$ and the cycle decomposition of~$\pi$ are the same partition of~$[n]$.

%
%

\section{Unimodal cycles and caterpillars}
\label{sec:edelman}

Edelman~\cite{edelman:on-inversions-a:} studied the minimum number of inversions among permutations with a fixed number of cycles. He showed that for $1\le k\le n$,
\[
	\min\{\ell(\sigma) : \text{$\sigma \in S_n$ and $c(\sigma) = k$}\} = n - k,
\]
and that this minimum is attained precisely by the direct sums of unimodal cycles of intervals. Since $\ell'(\pi) = n - c(\pi)$, achieving Edelman's minimum is the same as satisfying $\ell(\pi) = \ell'(\pi)$. Edelman's theorem therefore says that $\ell(\pi) = \ell'(\pi)$ if and only if~$\pi$ is a direct sum of unimodal cycles of intervals. In this section we recover this direct-sum description from the forest viewpoint, and as a structural bonus we show that every connected acyclic inversion graph is a caterpillar.

A \emph{unimodal cycle of an interval}%
\footnote{The terminology is potentially confusing: the unimodal cycles considered here are not the unimodal permutations studied elsewhere. A permutation $\pi = \pi_1 \cdots \pi_n$ is \emph{unimodal} if its one-line notation increases to its maximum and then decreases, equivalently if $\pi \in \mathrm{Av}(213, 312)$; these permutations are enumerated by $2^{n-1}$. This permutation class has been studied from several perspectives, including its cyclic structure and connections with one-dimensional dynamics by Gannon~\cite{gannon:the-cyclic-stru:}, its cycle enumerator by Thibon~\cite{thibon:the-cycle-enume:}, and its relationship with almost-increasing cycles by Archer and Lauderdale~\cite{archer:unimodal-permut:}, who give a bijection between almost-increasing cycles on~$[n]$ and unimodal permutations on $[n-1]$. The unimodal cycle terminology is also distinct from Chung's study of long unimodal subsequences in arbitrary permutations~\cite{chung:on-unimodal-sub:}.}
is a permutation of an interval of integers that can be written as a single cycle whose entries, after cyclic rotation, strictly increase to a maximum and then strictly decrease. After standardizing the interval to~$[n]$ and rotating the cycle to end with its minimum value (namely $1$ in the standardized form), the cycle has the form $(u_1, \ldots, u_r, n, d_1, \ldots, d_s, 1)$ where $u_1 < \cdots < u_r < n$ and $n > d_1 > \cdots > d_s > 1$. The cases $r = 0$ and $s = 0$ are allowed, as is the singleton cycle $(1)$, a fixed point.

For instance, take the permutation $\gamma = (3,4,5,8,9,7,6,2,1)$, written in cycle notation. The entries of this cycle increase to the maximum $9$ and then decrease to $1$, so $\gamma$ is a unimodal cycle of $[9]$. Its one-line notation, $\gamma = 314582697$, is not unimodal as a sequence itself. The inversion graph $G_\gamma$, shown in Figure~\ref{fig:unimodal-caterpillar}, is a caterpillar: its leaves are $1$, $4$, $5$, $6$, and $9$, and after removing them we obtain the path $3, 2, 8, 7$.

\begin{figure}[htbp]
\begin{center}
\begin{footnotesize}
	\begin{tikzpicture}[scale=0.42,baseline=(current bounding box.center)]
	\draw[darkgray,thick,rounded corners=0.01,line cap=round] (0,0) rectangle (10,10);
	\plotperm{3,1,4,5,8,2,6,9,7}
	\draw[dotted] (v3)--(v1);
	\draw[dotted] (v2)--(v4);
	\draw[dotted] (v2)--(v5);
	\draw[dotted] (v8)--(v6);
	\draw[dotted] (v7)--(v9);
	\draw (v3)--(v2)--(v8)--(v7);
	\labelnode[-45]{v1}{$1$}
	\labelnode[-45]{v2}{$2$}
	\labelnode[135]{v3}{$3$}
	\labelnode[135]{v4}{$4$}
	\labelnode[135]{v5}{$5$}
	\labelnode[-45]{v6}{$6$}
	\labelnode[-45]{v7}{$7$}
	\labelnode[135]{v8}{$8$}
	\labelnode[135]{v9}{$9$}
	\end{tikzpicture}
\end{footnotesize}
\end{center}
\caption{The inversion graph of the unimodal cycle $\gamma = (3,4,5,8,9,7,6,2,1) = 314582697$. The dotted edges are the leaf edges; removing them leaves the spine $3, 2, 8, 7$.}
\label{fig:unimodal-caterpillar}
\end{figure}
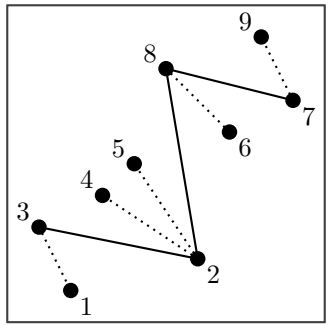

We record one useful feature of this normalization. In one-line notation, the position of a value is its predecessor in the cycle $(u_1, \ldots, u_r, n, d_1, \ldots, d_s, 1)$ above. It follows that the values $u_1, \ldots, u_r, n$ are precisely the \emph{left-to-right maxima}, meaning the entries larger than every entry to their left, while the values $d_1, \ldots, d_s, 1$ are the remaining entries; in fact, these remaining entries are \emph{right-to-left minima}, meaning entries smaller than every entry to their right. Thus, every entry of a unimodal cycle is either a left-to-right maximum or a right-to-left minimum. In particular, this shows immediately that unimodal cycles avoid $321$, because the middle entry of a copy of $321$ can be neither a left-to-right maximum nor a right-to-left minimum.

Sum components correspond to connected components of~$G_\pi$ by Proposition~\ref{prop:sum-components}, and~$G_\pi$ is a forest precisely when each of these components is a tree. The forest statement therefore reduces to the tree statement, and we treat only the connected case.

\subsection{Trees and unimodal cycles of intervals}
\label{subsec:trees-unimodal}

If~$G_\pi$ is a tree, then $\ell(\pi) = n - 1$ and $|\mathrm{supp}(\pi)| = n - 1$, so every reduced word for~$\pi$ uses each of $s_1, \ldots, s_{n-1}$ exactly once. The permutations whose reduced words use each simple reflection exactly once are the \emph{Coxeter elements} of~$S_n$, and the following proposition is the standard classification of Coxeter elements of~$S_n$ (and, by restriction, of the parabolic subgroup generated by $s_a, s_{a+1}, \ldots, s_{b-1}$). The result is well known; we include a proof because the inductive construction tracks the cycle structure explicitly, which we use below.

\begin{proposition}
\label{prop:coxeter-unimodal}
Let $a \le b$ be integers, and let $\gamma$ be a permutation of $[a, b]$. Then $\gamma$ is a unimodal cycle of $[a, b]$ if and only if it has a reduced word in which each of $s_a, s_{a+1}, \ldots, s_{b-1}$ appears exactly once.
\end{proposition}

\begin{proof}
We apply induction on $b - a$, removing the smallest value~$a$ from the cycle and the generator~$s_a$ from the word at each step.

If $b - a = 0$, then $[a, b] = \{a\}$ and $\gamma$ is the identity, the singleton unimodal cycle, with the empty reduced word.

If $b - a = 1$, then $\gamma$ is a unimodal cycle of $[a, b]$ if and only if $\gamma = (a+1, a) = s_a$, whose only reduced word, $s_a$, uses $s_a$ once. So the result holds.

For $b - a \ge 2$, suppose first that $\gamma$ has a reduced word using each of $s_a, \ldots, s_{b-1}$ exactly once. The generator $s_a$ commutes with every $s_j$ for $j \ge a + 2$, so among the letters of the word it fails to commute only with $s_{a+1}$. Since $s_{a+1}$ occurs just once, we may commute~$s_a$ past everything on one side of it, moving $s_a$ to the far left or the far right according to whether it occurs before or after~$s_{a+1}$:
\[
	\gamma = s_a \delta
	\qquad\text{or}\qquad
	\gamma = \delta s_a,
\]
where $\delta$ has a reduced word using each of $s_{a+1}, \ldots, s_{b-1}$ exactly once. By induction, $\delta$ is a unimodal cycle of $[a+1, b]$.

Since~$a$ is fixed by $\delta$, we may write $\delta$ in our normalization as $(v_1, \ldots, v_k, a+1)$, with the entries increasing to~$b$ and then decreasing to~$a+1$. Multiplying on the left gives, after rotation, $s_a \delta = (a+1, v_1, \ldots, v_k, a)$, a unimodal cycle of $[a, b]$ with~$a+1$ immediately following~$a$ in the cyclic order. Multiplying on the right gives $\delta s_a = (v_1, \ldots, v_k, a+1, a)$, again a unimodal cycle of~${[a, b]}$, this time with~$a+1$ immediately preceding~$a$.

Conversely, suppose $\gamma$ is a unimodal cycle of $[a, b]$, written
\[
	\gamma = (u_1, \ldots, u_r, b, d_1, \ldots, d_s, a).
\]
By unimodality, the value $a + 1$ is adjacent to~$a$ in the cyclic order: if it lies on the increasing side it equals $u_1$ and immediately follows~$a$, and if it lies on the decreasing side it equals $d_s$ and immediately precedes~$a$. Multiplying $\gamma$ by $s_a = (a, a+1)$ on the left in the first case, and on the right in the second case, removes~$a$ from the cycle and leaves a unimodal cycle of $[a+1, b]$. By induction, that smaller cycle has a reduced word using each of $s_{a+1}, \ldots, s_{b-1}$ exactly once, and reinserting $s_a$ on the appropriate side expresses $\gamma$ using each of $s_a, \ldots, s_{b-1}$ exactly once.

This expression is reduced. For each $i$ with $a \le i \le b - 1$, an expression in the simple reflections that avoids $s_i$ never swaps the entries in positions $i$ and $i + 1$, so it never moves a value across that cut, and hence the permutation it expresses preserves the set $[1, i]$. But $\gamma$ does not preserve $[1, i]$: it fixes $[1, a - 1]$ pointwise and is a single cycle on $[a, b]$, so it cannot preserve the proper nonempty subset $[a, i]$ of $[a, b]$. Every expression for $\gamma$ therefore contains each of $s_a, \ldots, s_{b-1}$, so $\ell(\gamma) \ge b - a$, and the constructed expression, of length $b - a$, is reduced.
\end{proof}

The tree statement follows.

\begin{proposition}
\label{prop:tree-unimodal}
An inversion graph~$G_\pi$ is a tree if and only if~$\pi$ is a unimodal cycle of~$[n]$.
\end{proposition}

\begin{proof}
Suppose~$G_\pi$ is a tree. Then $\ell(\pi) = |E(G_\pi)| = n - 1$ and $\kappa(\pi) = 1$, so $|\mathrm{supp}(\pi)| = n - 1$. Every reduced word for~$\pi$ has length $n - 1$ and uses $n - 1$ distinct simple reflections, so each of $s_1, \ldots, s_{n-1}$ appears exactly once. By Proposition~\ref{prop:coxeter-unimodal},~$\pi$ is a unimodal cycle of~$[n]$.

Conversely, suppose~$\pi$ is a unimodal cycle of~$[n]$. By Proposition~\ref{prop:coxeter-unimodal},~$\pi$ has a reduced word using each of $s_1, \ldots, s_{n-1}$ exactly once, so $\ell(\pi) = n - 1$ and hence $|E(G_\pi)| = n - 1$. Since~$\pi$ is a single~$n$-cycle, it is sum indecomposable, so~$G_\pi$ is connected by Proposition~\ref{prop:sum-components}. A connected graph on~$n$ vertices with $n - 1$ edges is a tree.
\end{proof}

Taking direct sums gives the following.

\begin{corollary}
\label{cor:forest-direct-sum-unimodal}
The inversion graph~$G_\pi$ is a forest if and only if~$\pi$ is a direct sum of unimodal cycles of intervals.
\end{corollary}

Proposition~\ref{prop:321:3412:forest} identifies forests with the permutations avoiding $321$ and $3412$, and Corollary~\ref{cor:forest-direct-sum-unimodal} identifies them with the direct sums of unimodal cycles of intervals; together these show that Edelman's minimum-inversion permutations are precisely the elements of $\mathrm{Av}(321, 3412)$.

%
%

\subsection{Connected acyclic inversion graphs and caterpillars}
\label{subsec:caterpillar}

A \emph{caterpillar} is a tree in which every nonleaf vertex lies on a single path, called the \emph{spine}; equivalently, removing all leaves yields a path.

We use the following elementary characterization. A tree~$T$ is a caterpillar if and only if its line graph has a Hamiltonian path; equivalently, if and only if the edges of~$T$ can be listed
\[
	e_1, e_2, \ldots, e_m
\]
so that consecutive edges share an endpoint. Indeed, if such an ordering is given with $m \ge 2$ (all trees with fewer edges are caterpillars), let $x_i$ be the endpoint shared by $e_i$ and $e_{i+1}$, and suppress consecutive repetitions from the sequence
\[
	x_1, x_2, \ldots, x_{m-1}.
\]
Any two consecutive terms of the resulting sequence are joined by an edge of~$T$, and since~$T$ has no cycles and no parallel edges, no vertex can recur, so the sequence is a path in~$T$. Every edge of~$T$ is incident with this path, so every vertex off the path is a leaf. The converse is immediate by listing the edges of a caterpillar along the spine, inserting the leaf edges incident with each spine vertex when that vertex is reached.

By Proposition~\ref{prop:tree-unimodal}, every connected acyclic inversion graph is the inversion graph of a unimodal cycle of~$[n]$, so it suffices to prove the caterpillar conclusion for inversion graphs of unimodal cycles. As discussed in the introduction, this result is not new; our contribution is the proof via unimodal cycles.

\begin{proposition}
\label{prop:unimodal-caterpillar}
If~$\pi$ is a unimodal cycle of~$[n]$, then~$G_\pi$ is a caterpillar.
\end{proposition}

\begin{proof}
The cases $n \le 2$ are immediate, so assume $n \ge 3$. By Proposition~\ref{prop:tree-unimodal},~$G_\pi$ is a tree, and~$\pi$ avoids $321$.

Let $L$ be the values of the left-to-right maxima of~$\pi$ and let $M$ be the values of the remaining entries, both listed from left to right. The list $L$ is increasing by definition, and $M$ is increasing because~$\pi$ avoids $321$. Because of this monotonicity, every edge of~$G_\pi$ joins a vertex of~$L$ to a vertex of~$M$.

For each $u \in L$, its neighbors in~$M$ form a consecutive block of the $M$-list. Indeed, suppose $u$ is adjacent to the entries $y'$ and $y$ of~$M$, with $y'$ before $y$. Then $u$ appears to the left of $y'$ and exceeds $y$, so every entry of~$M$ between $y'$ and $y$ in the list appears to the right of~$u$ and is smaller than~$u$, hence is also a neighbor of~$u$.

The key restriction is that these edges cannot cross: there do not exist $u$ before $u'$ in $L$ and $y'$ before $y$ in $M$ such that $uy$ and $u'y'$ are both edges of~$G_\pi$. If there were, then since both lists are increasing, $uy'$ and $u'y$ would also be inversions, and $u$, $y'$, $u'$, and $y$ would span a cycle in~$G_\pi$, contradicting that~$G_\pi$ is a tree. Consequently, if $u$ appears before $u'$ in $L$, then every neighbor of $u'$ occurs at or after every neighbor of~$u$ in the $M$-list, so the neighbor blocks of consecutive values of~$L$ either are disjoint or meet in a single entry of~$M$.

Consecutive blocks cannot be disjoint, since the cut between them would disconnect~$G_\pi$: earlier values of~$L$ have all their neighbors on one side of the cut, later values on the other. Hence, consecutive blocks meet in precisely one entry of~$M$.

List the edges of~$G_\pi$ block by block: first all edges incident with the first value of~$L$, in the order of their endpoints in~$M$; then all edges incident with the second value of~$L$; and so on. This lists every edge exactly once. Consecutive edges within a block share their endpoint in~$L$, and consecutive edges from different blocks share the entry of~$M$ where the two blocks meet. Thus, this edge ordering is a Hamiltonian path in the line graph of~$G_\pi$, and the characterization above implies that~$G_\pi$ is a caterpillar.
\end{proof}

The edge ordering in this proof can also be described by sorting the edges by the sum of their endpoint values: within a block the endpoint in~$L$ is fixed and the endpoint in~$M$ increases, while from one block to the next the shared endpoint in~$M$ is fixed and the endpoint in~$L$ increases, so the endpoint sums strictly increase along the path.

%
%

\section{Concluding remarks}
\label{sec:conclusion}

The cyclomatic number $\beta(G_\pi)$ of the inversion graph governs several apparently unrelated statistics of the permutation~$\pi$: every reduced word for~$\pi$ contains $\beta(G_\pi)$ repeated letters, the cycle surplus of~$\pi$ is at most $\beta(G_\pi)$, the length gap satisfies $\beta(G_\pi) \le \ell(\pi) - \ell'(\pi) \le 2\beta(G_\pi)$, and the total number of copies of $321$ and $3412$ in~$\pi$ is at least $\beta(G_\pi)$.

Reflection length also suggests a broader question: is there an analogue of this statistic for arbitrary simple graphs? In forthcoming work with Mandrick~\cite{mandrick:reflections-in-:}, which grew out of his thesis~\cite{mandrick:on-inversion-gr:} and in turn inspired the present work, we introduce an operation on simple graphs called \emph{edge reflection}, modeled on the effect that multiplying a permutation by a reflection has on its inversion graph, and study the resulting length statistic. Just as $\ell(\pi) - \ell'(\pi) = 0$ precisely when $G_\pi$ is a forest, the corresponding length gap in the graph setting vanishes precisely when the graph is a forest.

%
%

\bibliographystyle{abbrv}
\bibliography{references}

@unpublished{neuhaus:torus-actions-o:,
	author = {Neuhaus, Elke and Portakal, {\.I}rem and Paul, Niharika Chakrabarty},
	doi = {10.48550/arXiv.2510.00250},
	note = {arXiv:2510.00250 [math.AG]},
	title = {Torus Actions on Matrix {S}chubert and {K}azhdan--{L}usztig Varieties, and their Links to Statistical Models}}

@phdthesis{mandrick:on-inversion-gr:,
	author = {Mandrick, Sean},
	note = {arXiv:2506.22307 [math.CO]},
	school = {University of Florida},
	title = {On Inversion Graphs of Permutations},
	url = {https://ufdc.ufl.edu/UFE0062191/00001/pdf},
	year = {2025}}

@article{acan:on-random-trees:,
	author = {Acan, H{\"{u}}seyin and Hitczenko, Pawe\l},
	doi = {10.1016/j.disc.2016.05.028},
	fjournal = {Discrete Mathematics},
	issn = {0012-365X},
	journal = {Discrete Math.},
	mrclass = {05C05 (05A05 05C12 05C69 05C90)},
	mrnumber = {3533335},
	number = {12},
	pages = {2871--2883},
	title = {On random trees obtained from permutation graphs},
	volume = {339},
	year = {2016}}

@article{lee:toric-bruhat-in:,
	author = {Lee, Eunjeong and Masuda, Mikiya and Park, Seonjeong},
	doi = {10.1016/j.jcta.2020.105387},
	fjournal = {Journal of Combinatorial Theory. Series A},
	issn = {0097-3165,1096-0899},
	journal = {J. Combin. Theory Ser. A},
	mrclass = {52B05 (14M25)},
	mrnumber = {4190574},
	pages = {Paper No. 105387, 41 pp.},
	title = {Toric {B}ruhat interval polytopes},
	volume = {179},
	year = {2021}}

@article{gervacio:characterizatio:,
	author = {Gervacio, Severino V. and Rapanut, Teofina A. and Ramos, Phoebe Chloe F.},
	doi = {10.4236/ojdm.2013.31007},
	journal = {Open J. Discrete Math.},
	number = {1},
	pages = {33--38},
	title = {Characterization and Construction of Permutation Graphs},
	volume = {3},
	year = {2013}}

@unpublished{mandrick:reflections-in-:,
	author = {Mandrick, Sean and Tenner, Bridget Eileen and Vatter, Vincent},
	note = {In preparation},
	title = {Reflections in graphs}}

@book{golumbic:algorithmic-gra:,
	address = {Amsterdam, The Netherlands},
	author = {Golumbic, Martin Charles},
	edition = {Second},
	isbn = {0-444-51530-5},
	mrclass = {05C17 (05C85)},
	mrnumber = {2063679 (2005e:05061)},
	pages = {xxvi+314},
	publisher = {Elsevier},
	series = {Annals of Discrete Mathematics},
	title = {Algorithmic Graph Theory and Perfect Graphs},
	volume = {57},
	year = {2004}}

@article{dushnik:partially-order:,
	author = {Dushnik, Ben and Miller, Edwin Wilkinson},
	doi = {10.2307/2371374},
	journal = {Amer. J. Math.},
	pages = {600--610},
	title = {Partially ordered sets},
	volume = {63},
	year = {1941}}

@article{donten-bury:complexity-of-t:,
	author = {Donten-Bury, Maria and Escobar, Laura and Portakal, Irem},
	doi = {10.5802/alco.279},
	journal = {Algebr. Comb.},
	number = {3},
	pages = {835--861},
	title = {Complexity of the usual torus action on {K}azhdan--{L}usztig varieties},
	volume = {6},
	year = {2023}}

@article{gao:orbit-structure:fpsac,
	author = {Gao, Yibo and Hodges, Reuven},
	journal = {S\'em. Lothar. Combin.},
	pages = {Art. 105, 12 pp.},
	title = {Orbit structures and complexity in {S}chubert and {R}ichardson varieties (extended abstract)},
	volume = {93B},
	year = {2025}}

@article{brualdi:permutation-gra:,
	author = {Brualdi, Richard A. and Dahl, Geir},
	doi = {10.26493/2590-9770.1536.8db},
	journal = {Art Discrete Appl. Math.},
	number = {3},
	pages = {Paper No. 3.01, 20 pp.},
	title = {Permutation graphs and the weak {B}ruhat order},
	volume = {6},
	year = {2023}}

@article{thibon:the-cycle-enume:,
	author = {Thibon, Jean-Yves},
	doi = {10.1007/s00026-001-8024-6},
	fjournal = {Annals of Combinatorics},
	issn = {0218-0006},
	journal = {Ann. Comb.},
	mrclass = {05A15 (05A05 05E05)},
	mrnumber = {1897638},
	number = {3-4},
	pages = {493--500},
	title = {The cycle enumerator of unimodal permutations},
	volume = {5},
	year = {2001}}

@article{karuppuchamy:on-schubert-var:,
	author = {Karuppuchamy, Paramasamy},
	doi = {10.1080/00927872.2011.635620},
	journal = {Comm. Algebra},
	number = {4},
	pages = {1365--1368},
	title = {On {S}chubert varieties},
	volume = {41},
	year = {2013}}

@article{matsumoto:generateurs-et-:,
	author = {Matsumoto, Hideya},
	journal = {C. R. Acad. Sci. Paris},
	pages = {3419--3422},
	title = {G\'en\'erateurs et relations des groupes de {W}eyl g\'en\'eralis\'es},
	volume = {258},
	year = {1964}}

@article{daly:reduced-decompo:,
	author = {Daly, Daniel},
	doi = {10.1007/s00373-011-1100-8},
	journal = {Graphs Combin.},
	number = {2},
	pages = {173--185},
	title = {Reduced decompositions with one repetition and permutation pattern avoidance},
	volume = {29},
	year = {2013}}

@article{tenner:repetition-in-r:,
	author = {Tenner, Bridget Eileen},
	doi = {10.1016/j.aam.2012.01.002},
	journal = {Adv. in Appl. Math.},
	number = {1},
	pages = {1--14},
	title = {Repetition in reduced decompositions},
	volume = {49},
	year = {2012}}

@incollection{gallai:a-translation-o:,
	address = {Chichester, England},
	author = {Maffray, Fr{\'e}d{\'e}ric and Preissmann, Myriam},
	booktitle = {Perfect Graphs},
	editor = {{Ram{\'\i}rez Alfons{\'\i }n}, Jorge Luis and Reed, Bruce Alan},
	pages = {25--66},
	publisher = {Wiley},
	series = {Wiley Series in Discrete Math. \& Optim.},
	title = {A translation of {G}allai's paper: ``{T}ransitiv orientierbare {G}raphen''},
	volume = {44},
	year = {2001}}

@article{gallai:transitiv-orien:,
	author = {Gallai, Tibor},
	fjournal = {Acta Mathematica Academiae Scientiarum Hungaricae},
	issn = {0001-5954},
	journal = {Acta Math. Acad. Sci. Hungar.},
	mrnumber = {MR0221974 (36 \#5026)},
	number = {1-2},
	pages = {25--66},
	title = {Transitiv orientierbare {G}raphen},
	url = {https://doi.org/10.1007/BF02020961},
	volume = {18},
	year = {1967}}

@article{pnueli:transitive-orie:,
	author = {Pnueli, Amir and Lempel, Abraham and Even, Shimon},
	doi = {10.4153/CJM-1971-016-5},
	fjournal = {Canadian Journal of Mathematics. Journal Canadien de Math\'{e}matiques},
	issn = {0008-414X},
	journal = {Canad. J. Math.},
	mrclass = {05C20},
	mrnumber = {0292717},
	pages = {160--175},
	title = {Transitive orientation of graphs and identification of permutation graphs},
	volume = {23},
	year = {1971}}

@incollection{vatter:permutation-cla:,
	address = {Boca Raton, Florida},
	author = {Vatter, Vincent},
	booktitle = {Handbook of Enumerative Combinatorics},
	editor = {B{\'o}na, Mikl{\'o}s},
	pages = {754--833},
	publisher = {CRC Press},
	title = {Permutation classes},
	year = {2015}}

@incollection{brandstadt:on-the-restrict:,
	address = {Berlin, West Germany},
	author = {Brandst\"adt, Andreas and Kratsch, Dieter},
	booktitle = {Fundamentals of computation theory ({C}ottbus, 1985)},
	doi = {10.1007/BFb0028791},
	isbn = {3-540-15689-5},
	mrclass = {68Q25 (05C75)},
	mrnumber = {821224},
	pages = {53--62},
	publisher = {Springer},
	series = {Lecture Notes in Comput. Sci.},
	title = {On the restriction of some {NP}-complete graph problems to permutation graphs},
	volume = {199},
	year = {1985}}

@article{archer:unimodal-permut:,
	author = {Archer, Kassie and Lauderdale, L.-K.},
	doi = {10.37236/6954},
	fjournal = {Electronic Journal of Combinatorics},
	issn = {1077-8926},
	journal = {Electron. J. Combin.},
	mrclass = {05A05 (05A15 05A19)},
	mrnumber = {3711105},
	number = {3},
	pages = {Paper \#P3.63, 14 pp.},
	title = {Unimodal permutations and almost-increasing cycles},
	volume = {24},
	year = {2017}}

@article{chung:on-unimodal-sub:,
	author = {Chung, F. R. K.},
	doi = {10.1016/0097-3165(80)90021-7},
	issn = {0097-3165},
	journal = {J. Combin. Theory Ser. A},
	mrclass = {05A05},
	mrnumber = {MR600589 (82a:05002)},
	number = {3},
	pages = {267--279},
	title = {On unimodal subsequences},
	volume = {29},
	year = {1980}}

@article{gannon:the-cyclic-stru:,
	author = {Gannon, Terry},
	doi = {10.1016/S0012-365X(00)00368-X},
	fjournal = {Discrete Mathematics},
	issn = {0012-365X},
	journal = {Discrete Math.},
	mrclass = {05A05},
	mrnumber = {1835657},
	number = {1-3},
	pages = {149--161},
	title = {The cyclic structure of unimodal permutations},
	volume = {237},
	year = {2001}}

@article{edelman:on-inversions-a:,
	author = {Edelman, Paul H.},
	doi = {10.1016/S0195-6698(87)80031-8},
	fjournal = {European Journal of Combinatorics},
	issn = {0195-6698,1095-9971},
	journal = {European J. Combin.},
	mrclass = {05A15 (05A05 05A20)},
	mrnumber = {919878},
	number = {3},
	pages = {269--279},
	title = {On inversions and cycles in permutations},
	volume = {8},
	year = {1987}}

@article{petersen:the-depth-of-a-:,
	author = {Petersen, Thomas Kyle and Tenner, Bridget Eileen},
	doi = {10.4310/JOC.2015.v6.n1.a9},
	fjournal = {Journal of Combinatorics},
	issn = {2156-3527},
	journal = {J. Comb.},
	mrclass = {20B30 (05A05 05A15)},
	mrnumber = {3338848},
	number = {1-2},
	pages = {145--178},
	title = {The depth of a permutation},
	volume = {6},
	year = {2015}}

@article{tenner:pattern-avoidan:,
	author = {Tenner, Bridget Eileen},
	doi = {10.1016/j.jcta.2006.10.003},
	fjournal = {Journal of Combinatorial Theory. Series A},
	issn = {0097-3165},
	journal = {J. Combin. Theory Ser. A},
	mrclass = {05E15 (06A07 20F55)},
	mrnumber = {2333139},
	number = {5},
	pages = {888--905},
	title = {Pattern avoidance and the {B}ruhat order},
	volume = {114},
	year = {2007}}

\end{document}